\documentclass[11pt]{amsart}
\usepackage{amssymb}
\usepackage{verbatim}
\usepackage{hyperref}
\usepackage{color}

\begin{document}

\newtheorem{theorem}{Theorem}    
\newtheorem{proposition}[theorem]{Proposition}
\newtheorem{conjecture}[theorem]{Conjecture}
\def\theconjecture{\unskip}
\newtheorem{corollary}[theorem]{Corollary}
\newtheorem{lemma}[theorem]{Lemma}
\newtheorem{sublemma}[theorem]{Sublemma}
\newtheorem{fact}[theorem]{Fact}
\newtheorem{observation}[theorem]{Observation}
\theoremstyle{definition}
\newtheorem{definition}{Definition}
\newtheorem{notation}[definition]{Notation}
\newtheorem{remark}[definition]{Remark}
\newtheorem{question}[definition]{Question}
\newtheorem{questions}[definition]{Questions}
\newtheorem{example}[definition]{Example}
\newtheorem{problem}[definition]{Problem}
\newtheorem{exercise}[definition]{Exercise}

\numberwithin{theorem}{section}
\numberwithin{definition}{section}
\numberwithin{equation}{section}

\def\reals{{\mathbb R}}
\def\torus{{\mathbb T}}
\def\heis{{\mathbb H}}
\def\integers{{\mathbb Z}}
\def\rationals{{\mathbb Q}}
\def\naturals{{\mathbb N}}
\def\complex{{\mathbb C}\/}

\def\distance{\operatorname{distance}\,}
\def\support{\operatorname{support}\,}
\def\dist{\operatorname{dist}\,}
\def\Span{\operatorname{span}\,}
\def\degree{\operatorname{degree}\,}
\def\kernel{\operatorname{kernel}\,}
\def\dim{\operatorname{dim}\,}
\def\codim{\operatorname{codim}}
\def\trace{\operatorname{trace\,}}
\def\Span{\operatorname{span}\,}
\def\dimension{\operatorname{dimension}\,}
\def\codimension{\operatorname{codimension}\,}
\def\nullspace{\scriptk}
\def\kernel{\operatorname{Ker}}
\def\ZZ{ {\mathbb Z} }
\def\p{\partial}
\def\rp{{ ^{-1} }}
\def\Re{\operatorname{Re\,} }
\def\Im{\operatorname{Im\,} }
\def\ov{\overline}
\def\eps{\varepsilon}
\def\lt{L^2}
\def\diver{\operatorname{div}}
\def\curl{\operatorname{curl}}
\def\etta{\eta}
\newcommand{\norm}[1]{ \|  #1 \|}
\def\expect{\mathbb E}
\def\bull{$\bullet$\ }
\def\det{\operatorname{det}}
\def\Det{\operatorname{Det}}
\def\multiR{\mathbf R}
\def\bestA{\mathbf A}
\def\Apq{\mathbf A_{p,q}}
\def\Apqr{\mathbf A_{p,q,r}}
\def\rank{\mathbf r}
\def\diameter{\operatorname{diameter}}

\def\essd{\operatorname{essential\ diameter}}

\newcommand{\abr}[1]{ \langle  #1 \rangle}

\def\doublesymm{\naturalGl}

\newcommand{\Norm}[1]{ \Big\|  #1 \Big\| }
\newcommand{\set}[1]{ \left\{ #1 \right\} }
\def\one{{\mathbf 1}}
\newcommand{\modulo}[2]{[#1]_{#2}}

\def\barrier{\medskip\hrule\hrule\medskip}

\def\bb{\mathbb B}
\def\br{\mathbf{r}}
\def\bt{\mathbf{t}}
\def\bE{\mathbf{E}}
\def\EE{{\mathbf E}\/}
\def\symdif{\,\Delta\,}
\def\sstar{{\dagger\star}}
\def\Star{{\bullet}}
\def\aff{\operatorname{Aff}}
\def\gl{\operatorname{Gl}}
\def\baff{\operatorname{\bf Aff}}
\newcommand{\mbf}[1]{\mathbf #1}

\def\repair{\bigskip\hrule\hrule\bigskip}

\def\scriptf{{\mathcal F}}
\def\scripts{{\mathcal S}}
\def\scriptq{{\mathcal Q}}
\def\scriptg{{\mathcal G}}
\def\scriptm{{\mathcal M}}
\def\scriptb{{\mathcal B}}
\def\scriptc{{\mathcal C}}
\def\scriptt{{\mathcal T}}
\def\scripti{{\mathcal I}}
\def\scripte{{\mathcal E}}
\def\scriptv{{\mathcal V}}
\def\scriptw{{\mathcal W}}
\def\scriptu{{\mathcal U}}
\def\scripta{{\mathcal A}}
\def\scriptr{{\mathcal R}}
\def\scripto{{\mathcal O}}
\def\scripth{{\mathcal H}}
\def\scriptd{{\mathcal D}}
\def\scriptl{{\mathcal L}}
\def\scriptn{{\mathcal N}}
\def\scriptp{{\mathcal P}}
\def\scriptk{{\mathcal K}}
\def\scriptP{{\mathcal P}}
\def\scriptj{{\mathcal J}}
\def\scriptz{{\mathcal Z}}
\def\frakv{{\mathfrak V}}
\def\frakG{{\mathfrak G}}
\def\frakA{{\mathfrak A}}
\def\frakB{{\mathfrak B}}
\def\frakC{{\mathfrak C}}

\author{Michael Christ}
\address{
        Michael Christ\\
        Department of Mathematics\\
        University of California \\
        Berkeley, CA 94720-3840, USA}
\email{mchrist@math.berkeley.edu}
\thanks{Research supported by NSF grant DMS-1363324.}


\date{May 30, 2015.}

\title[Near equality  in the Riesz-Sobolev inequality]
{Near equality \\ in the Riesz-Sobolev inequality  \\ in higher dimensions}

\begin{abstract}
The Riesz-Sobolev inequality  provides an upper bound
for a trilinear expression involving convolution of indicator functions of sets.
It is known that equality holds only for homothetic ordered triples of appropriately
situated ellipsoids. 
We characterize ordered triples of subsets of Euclidean space $\reals^d$ 
that nearly realize equality, for arbitrary dimensions $d$, 
extending a result already known for $d=1$.
\end{abstract}
\maketitle

\section{Introduction}

Let $d\ge 1$.
Let $|A|$ denote the Lebesgue measure of a set $A\subset\reals^d$.
For any Lebesgue measurable set $E\subset\reals^d$ satisfying
$0\le |E|<\infty$
define $E^\Star$ to be the (nonempty) closed ball centered at $0$ 
satisfying $|E^\Star|=|E|$.
For any ordered triple $EE=(E_1,E_2,E_3)=(E_j)_{1\le j\le 3}$ of subsets of $\reals^d$ 
define $\EE^\Star= (E_j^\Star)_{1\le j\le 3}$.
Denote by $A\symdif B$ the symmetric difference of sets $A,B$.

The Riesz-Sobolev inequality \cite{liebloss} concerns the quantity 
\begin{equation} 
\scriptt(\EE) = \int _{(\reals^d)^3} \prod_{j=1}^3 \one_{E_j}(x_j)\,d\lambda(x)
\end{equation}
where $\lambda$ is the natural Lebesgue measure on $\{(x_1,x_2,x_3)\in (\reals^d)^3:
x_1+x_2+x_3=0\}$. 
It states that for any $\EE$,
\begin{equation}\label{eq:RS} \scriptt(\EE)\le \scriptt(\EE^\Star).  \end{equation}
By replacing $E_3$ by its reflection about the origin,
this can equivalently be written 
\begin{equation} \int_{E_3}\one_{E_1}*\one_{E_2} 
\le \int_{E_3^\Star}\one_{E_1^\Star}*\one_{E_2^\Star}.
\end{equation}

Burchard \cite{burchard} has characterized triples $\EE$ that achieve equality in
\eqref{eq:RS}.  An ordered triple $\br=(r_1,r_2,r_3)$ of positive real numbers
is said to be admissible if $r_k\le r_i+r_j$ for each permutation 
$(i,j,k)$ of $(1,2,3)$, and to be strictly admissible if
$r_k < r_i+r_j$ for each permutation. 
An ordered triple $\EE$ of Lebesgue measurable subsets of $\reals^d$ is said to be strictly
admissible if $(|E_j|^{1/d})_{1\le j\le 3}$ is a strictly admissible ordered
triple of positive numbers.
Burchard's characterization states that for strictly admissible ordered
triples, equality holds if and only if there exist an ellipsoid $\scripte\subset\reals^d$,
vectors $v_j\in\reals^d$ satisfying $v_1+v_2+v_3=0$, and $r_j\in\reals^+$
such that 
\begin{equation} |E_j\symdif \big(r_j \scripte+v_j\big)|=0.\end{equation}
A less restrictive characterization, involving homothetic convex sets 
in place of ellipsoids, applies in the borderline case of nonstrict admissibility,
but is not of direct relevance here.

This paper is concerned with a characterization of near equality, that is,
of triples that satisfy the reverse inequality
$\scriptt(\EE)\ge(1-\delta)\scriptt(\EE^\Star)$ with $\delta$ small.
The case $d=1$ was treated in \cite{christRS3}. Here we extend the result
obtained there to arbitrary dimensions, albeit with a less quantitative formulation.

We will work in the context of a quantitative form of strict admissibility.
Let $\tau\in(0,\infty)$.  An ordered triple $\br=(r_j)_{1\le j\le 3}$ 
of positive real numbers is said to be $\tau$--admissible if 
\begin{equation} 
r_i + r_j \ge r_k + \tau\max(r_1,r_2,r_3) \end{equation}
for all permutations $(i,j,k)$ of $(1,2,3)$.
An ordered triple $\EE$ of Lebesgue measurable subsets of $\reals^d$ 
with positive, finite measures
is $\tau$--admissible if 
$(|E_n|^{1/d})_{1\le n\le 3}$ is a $\tau$--admissible ordered
triple of positive numbers.

\begin{theorem} \label{thm:RShighermain}
Let $d\ge 1$.
For every $\tau>0$ and $\eps>0$ there exists $\delta>0$ with the following property.
Let $\EE$ be a $\tau$--admissible ordered triple of Lebesgue measurable 
subsets of $\reals^d$ with positive, finite Lebesgue measures.
Suppose that $\scriptt(\EE)\ge (1-\delta)\scriptt(\EE^\Star)$.
Then there exist an ellipsoid $\scripte\subset\reals^d$ centered at
$0$, 
along with vectors $v_j\in\reals^d$ satisfying $v_1+v_2+v_3=0$,
such that
\begin{equation} \label{eq:mainconclusion}
|E_j \symdif \big(r_j\scripte+v_j\big)|\le \eps |E_j|
\ \ \text{for each index $j$,}
\end{equation}
where $r_j = |E_j|^{1/d}|\scripte|^{-1/d}$.
\end{theorem}

The present paper is an essential step in a larger project.
The analysis developed below
involves a compactness step, and consequently yields no information 
concerning the dependence of $\delta$ on $\eps$.
For $d=1$, the optimal dependence $\eps=O(\delta^{1/2})$ for fixed $\tau$
was proved in \cite{christRS3}. 
We intend to establish a bound of this same type for all dimensions in a future work.
While that result will formally supersede the main result of this paper,
its proof will rely on the conclusion of Theorem~\ref{thm:RShighermain}
as input, rather than enhancing or replacing the proof developed here. 

The analysis in \cite{christRS3} was based 
on an inverse theorem of additive combinatorics.
We have not been able to extend that same method of proof to higher dimensions; 
nor do we introduce here any alternative approach to the one-dimensional case.
Instead, we exploit ideas connected with symmetrization to argue by induction on
the dimension $d$,
using the one-dimensional result in a ``black box'' spirit. 
Additive combinatorial considerations do intervene, but play a secondary role 
in the induction step.
A central technique, already present in \cite{burchard} and \cite{christbmhigh}, 
is the exploitation of partially symmetrized sets intermediate
between an arbitrary $E\subset\reals^d$ and its fully symmetrized partner $E^\Star$.

\section{Notations}

Let $d\ge 1$. 
Let $\bb_d$ be the closed unit ball in $\reals^d$,
and let $\omega_d=|\bb_d|$ be its Lebesgue measure. 
Let \begin{equation} \phi(t) = (1-|t|^2)^{1/2}\ \text{for $|t|\le 1$.}\end{equation}

$\reals^+=(0,\infty)$ is the set of all strictly positive real numbers.
$\tau$--admissibility of an ordered triple $\br=(r_1,r_2,r_3)$ of positive real numbers,
and of an ordered triple ${\mathbf E}=(E_1,E_2,E_3)$ of subsets of $\reals^d$,
are defined above. 
By permuting the indices so that $r_i\le r_j\le r_k$ one finds that
if $\br$ is $\tau$--admissible,
\begin{equation} \min_{\kappa\in\{1,2,3\}}r_\kappa 
\ge\tau \max_{\kappa\in\{1,2,3\}}r_\kappa. \end{equation}

The expression $\scriptt({\mathbf E}) =
\int_{x_1+x_2+x_3=0} \prod_{j=1}^3 \one_{E_j}\,d\lambda(x)$ was likewise introduced above.
It will often be convenient to reduce to an alternative expression
which involves three sets of equal measures.
This is achieved by introducing an ordered triple $\br\in(0,\infty)^3$
and considering $\scriptt(r_1A_1,r_2A_2,r_3A_3)$
where $A_j\subset\reals^d$, $|A_1|=|A_2|=|A_3|=|\bb_d|$ and $rA=\{rx: x\in A\}$.

\begin{definition} An ordered triple $\mbf{E}$ 
of Lebesgue measurable subsets of $\reals^d$ with positive, finite measures
is a $\delta$--near extremizer of the Riesz-Sobolev inequality \eqref{eq:RS} if
\begin{equation} \scriptt(\mbf{E})\ge (1-\delta)\scriptt(\mbf{E}^\Star).\end{equation} 
\end{definition}

The symmetrization $E^\Star$ of $E\subset\reals^d$ was defined above.
Partial symmetrizations $E^\dagger,E^\star,E^\sstar = (E^\dagger)^\star$, used in the analysis, 
are defined in \S\ref{section:symms}.

We often identify $\reals^d$ with $\reals^{d-1}\times\reals^1$
and for any set $E\subset\reals^d$ and any
$x'\in\reals^{d-1}$ and $s\in\reals$ we set
\begin{align} E^{x'}&=\set{t\in\reals: (x',t)\in E} 
\\ E_{s}&=\set{y'\in\reals^{d-1}: (y',s)\in E}. \end{align}

Denote by $\pi(E)$ the projection
\begin{equation} \pi(E)=\set{x'\in\reals^{d-1}: E^{x'}\ne\emptyset}.  \end{equation}

For any $\br\in (0,\infty)^3$
consider 
\begin{align*}
S_\br&=\{(x_1,x_2,x_3)\in (\reals^d)^3: \sum_{j=1}^3 r_j x_j=0\}
\\ S^{x_1}_\br&=\{(x_2,x_3)\in (\reals^d)^2: (x_1,x_2,x_3)\in S_\br\}.
\end{align*}
There are natural measures on these sets,
defined by Lebesgue measures of projections onto $(\reals^d)^2$
and onto $\reals^d$, respectively,
using projections $(x_1,x_2,x_3)\mapsto (x_i,x_j)$ for any $i\ne j$ for $S$,
and projections $(x_2,x_3)\mapsto x_i$ for any $i\in\{2,3\}$ for $S^{x_1}$.
These measures are independent of the choices of indices. 
We denote them by $\lambda_\br$ and by $\lambda^{x_1}_\br$, respectively.
We will frequently allow the parameter $\br$ to be understood,
and will simply denote these sets and measures by $S,S^{x_1},\lambda,\lambda^{x_1}$ respectively.

Introduce the function 
\begin{equation} \Lambda_d(\gamma_1,\gamma_2,\gamma_3) = \scriptt(B_1,B_2,B_3) \end{equation}
where $B_j$ is the closed ball in $\reals^d$ centered at $0$ of measure $\gamma_j$.
In these terms, the Riesz-Sobolev inequality states that 
\begin{equation} \scriptt(E_1,E_2,E_3) \le \Lambda_d(|E_1|,|E_2|,|E_3|). \end{equation}

\begin{definition} 
$\aff(d)$ denotes the group of all invertible affine automorphisms of $\reals^d$.
These are of the form $x\mapsto T(x)=A(x)+v$ where $A$ is an invertible linear
automorphism of $\reals^d$ and $v\in\reals^d$.  \end{definition}

\begin{definition} $\baff(d)$ denotes the set of all ordered triples
${\mathbf T} = (T_1,T_2,T_3)$ of elements \[T_j(x)=A(x)+v_j(x) \]
of $\aff(d)\times\aff(d)\times\aff(d)$ where $A$ is an element of the general
linear group $\gl(d)$ that is independent of $j$, and $v_j\in\reals^d$ satisfy
\begin{equation} v_1+v_2+v_3=0.  \end{equation} \end{definition}

$\baff(d)$ is a subgroup of the product group $\aff(d)\times\aff(d)\times\aff(d)$. 
We write $\mbf{T} = (T_1,T_2,T_3)$ and
\begin{equation}{\mathbf T}({\mathbf E}) = (T_1(E_1),\,T_2(E_2),\,T_3(E_3)).\end{equation}
The trilinear form $\scriptt$ satisfies
\begin{equation}
\scriptt({\mathbf T}({\mathbf E})) = |\det(A)|^2 \, \scriptt({\mathbf E})
\end{equation}
for all ${\mathbf E}$, where $A=A_1=A_2=A_3$.
In this sense, $\baff(d)$ is a group of symmetries of $\scriptt$.
In particular, $\mbf{T}(\EE)$ is a $\delta$--near extremizer
of the Riesz-Sobolev inequality if and only if $\EE$ is so.
Likewise, $\mbf{T}(\EE)$ is $\tau$--admissible if and only if $\EE$ is so.

\section{Preparations}
Regard $\reals^d$ as $\reals^{d-k}\times\reals^k$ where $k\in\{1,2,\dots,d-1\}$.
Let $\pi:\reals^d\to\reals^{d-k}$ denote the projection onto the first factor; $\pi(x',x'')=x'$. 

\begin{lemma} \label{lemma:allplayball}
Let $\br\in(\reals^+)^3$ be strictly admissible.
For every $y_3\in \reals^{d-k}$ satisfying $|y_3|<r_3$
the set $S(y_3)$ of all $(y_1,y_2)\in\reals^{d-k}\times\reals^{d-k}$
satisfying $|y_j|<r_j$ and $y_1+y_2+y_3=0$
such that $((r_j^2-|y_j|^2)^{1/2})_{1\le j\le 3}$ is strictly admissible
is nonempty and has positive $d-k$--dimensional Lebesgue measure.
Moreover, the $d-k$--dimensional Lebesgue measure of 
$S(y_3)$ is a lower semicontinuous function of $y_3,\br$.
\end{lemma}

\begin{proof}
For any $s_3\in\reals$
satisfying $|s_3|<r_3$
there exist $s_1,s_2\in\reals$ 
satisfying $|s_j|<r_j$ and $\sum_{j=1}^3 s_j=0$ such that
$((r_j^2-s_j^2)^{1/2})_{1\le j\le 3}$ is strictly admissible.
This is Lemma~7.1 of Burchard \cite{burchard}.
It follows immediately that for $\br$ strictly admissible
and $y_3\in\reals^{d-1}$ satisfying $|y_3|<r_3$
there exist $y_1,y_2\in\reals^{d-1}$ such that
$|y_j|<r_j$, $\sum_{j=1}^3 y_j=0$, and 
$((r_j^2-|y_j|^2)^{1/2})_{1\le j\le 3}$ is strictly admissible;
apply the preceding statement with $s_3=|y_3|$
and set $y_j=s_j y_3/s_3$ for $j=1,2$.
The remaining conclusions follow from continuity. 
\end{proof}

Continue to express $\reals^d = \reals^{n-k}\times\reals^k$ with $0<k<d$,
with coordinates $x=(x',x'')$.

\begin{lemma}\label{lemma:paininneck}
Let $d\ge 2$ and $k\in\{1,2,\dots,d-1\}$.  
For every $\tau,\eps>0$ there exists $\eta>0$
with the following property.
For any $\tau$--admissible $\br\in(\reals^+)^3$, 
$\bb_{d-k}$ can be measurably partitioned as
$\bb_{d-k}= \scriptg\cup \scriptb$
with $|\scriptb|<\eps$ and
for every $x'_3\in\scriptg$,
\[ \lambda_{x'_3} \big\{(x'_1,x'_2): \sum_{j=1}^3 r_j x'_j=0
\text{ and } (r_j \phi(x'_j)) \text{ is $\eta$--admissible } \big\} \ge\eta.\]
\end{lemma} 

\begin{proof}
We may assume without loss of generality that $\max(r_1,r_2,r_3)=1$.
The set of all $\tau$--admissible $\br$ satisfying this normalization
is a compact subset of $(\reals^+)^3$.
The lemma is consequently a simple consequence of Lemma~\ref{lemma:allplayball}
by a compactness argument, working with the open unit ball in $\reals^{d-k}$
replaced by a closed ball $\{x'\in\reals^{d-k}: |x'|\le 1-\epsilon\}$
where $\epsilon$ depends on $\eps,\tau$.
\end{proof}

\begin{lemma} \label{lemma:zeroprime}
Let $d\ge 2$ and $k\in\{1,2,\dots,d-1\}$.
For any $\tau,\eps>0$ there exists $\delta>0$ with the following property.
Let $\br$ be $\tau$--admissible. 
Let  $\bE$ be an ordered triple of Lebesgue measurable subsets of $\reals^d$ satisfying
\begin{equation} \big| E_j\symdif r_j\bb_d \big| <\delta \end{equation}
for each $j\in\{1,2,3\}$.
Then there exists a partition $E_1 = \scriptg\cup \scriptb$ such that 
\begin{equation} |\scriptb|<\eps \end{equation}
and for each $y_1\in \pi(\scriptg)$,
the set $\tilde \scriptg$ of all pairs $(y_2,y_3)\in \pi(E_2)\times\pi(E_3)$
satisfying $y_1+y_2+y_3=0$ such that
$(|\scriptg_1^{y_1}|,\,|E_2^{y_2}|,\,|E_3^{y_3}|)$ is $\delta$--admissible
satisfies
\[ \lambda_{y_1}(\tilde \scriptg) \ge \delta |\pi(E_1)|.  \]
\end{lemma}

This is a direct consequence of Lemma~\ref{lemma:paininneck}
since $E_j$ and $r_j\bb_d$  have small symmetric difference.
\qed

The next lemma is very simple. We include a conceptual (rather than algebraic)
proof for completeness.
\begin{lemma}
Let $d\ge 1$.  Let $\alpha,\beta\in[0,\infty)^3$ and $\gamma\in(0,\infty)^3$ 
satisfy $\gamma=\alpha+\beta$, and suppose that
$\gamma$ is admissible.  Then 
\begin{equation} \Lambda_d(\mbf{\alpha})+\Lambda_d(\mbf{\beta})\le \Lambda_d(\mbf{\gamma}).\end{equation}
Moreover, if  $\Lambda_d(\mbf{\alpha})+\Lambda_d(\mbf{\beta}) = \Lambda_d(\mbf{\gamma})$
then $\mbf{\alpha}=0$ or $\mbf{\beta}=0$.
\end{lemma}

\begin{proof}
Choose any three distinct nonzero points $y_j\in\reals^d$ satisfying $\sum_{j=1}^3 y_j=0$.
Let $z_j=ry_j$ where $r>0$ is a large quantity to be chosen below.
Let $B'_j$ be balls centered at $0$ of measures $\alpha_j$
and let $B''_j$ be balls centered at $z_j\in\reals^d$ of measures $\beta_j$.
Then
$\scriptt(B'_1,B'_2,B'_3)=\Lambda_d(\mbf{\alpha})$, and
$\scriptt(B''_1,B''_2,B''_3)=\Lambda_d(\mbf{\beta})$.

Set $E_j=B'_j\cup B''_j$. 
By expressing $\one_{E_j} = \one_{B'_j}+\one_{B''_j}$,
invoking multilinearity of $\scriptt$ and expanding, we express
$\scriptt(E_1,E_2,E_3)$ as a sum of eight terms, each of
which equals $\scriptt(A_1,A_2,A_3)$ where $A_j$ is either $B'_j$ or $B''_j$.
If $r$ is chosen to be sufficiently large then all but two of these terms vanish, leaving
\[\Lambda_d(\gamma)\ge \scriptt(E_1,E_2,E_3) =\scriptt(B'_1,B'_2,B'_3) +\scriptt(B''_1,B''_2,B''_3)
= \Lambda_d(\alpha)+\Lambda_d(\beta).\]

If $\Lambda_d(\mbf{\alpha})+\Lambda_d(\mbf{\beta})
= \Lambda_d(\mbf{\gamma})$ then this chain of inequalities forces
$\scriptt(E_1,E_2,E_3) = \Lambda_d(\mbf{\gamma})$.
By Burchard's theorem,
each of the sets $E_j$ differs from some convex set by a set of Lebesgue measure zero.
But for large $r$, $B'_j\cap B''_j=\emptyset$. This forces either $B'_j$ or $B''_j$ to have radius zero.

It must be that $|B'_j|=0$ for all $j$, or $|B''_j|=0$ for all $j$.
For if some $B'_j$ and $B''_i$ both have measure zero then both $\scriptt(B'_1,B'_2,B'_3)$
and $\scriptt(B''_1,B''_2,B''_3)$ vanish, contradicting the assumption of equality
since $\Lambda_d(\mbf{\gamma})>0$.
\end{proof}

\begin{corollary} \label{cor:Lambdastrongtriangle}
Let $d\ge 1$ and $\tau,\eta>0$.
There exists $\rho>0$ such that
for any $\tau$--admissible ordered triple ${\gamma}\in(0,\infty)^3$
and any decomposition ${\gamma} = {\alpha}+{\beta}$
with ${\alpha},{\beta}\in[0,\infty)^3$
and 
\begin{align}
\max_{1\le j\le 3} \alpha_j \ge \eta \min_{1\le i\le 3}\gamma_i
\ \text{ and }\  \max_{1\le j\le 3} \beta_j \ge \eta \min_{1\le i\le 3}\gamma_i,
\end{align}
there is a uniformly strict inequality
\begin{equation}
\Lambda_d({\alpha})+\Lambda_d({\beta}) \le (1-\rho) \Lambda_d({\gamma}).
\end{equation}
\end{corollary}

\begin{proof}
Let $d,\tau,\eta$ be given. 
Because the hypotheses and conclusions are invariant under multiplication
of all $\alpha_i,\beta_j,\gamma_k$ by any common  positive constant, we
may assume without loss of generality that $\max_{1\le i\le 3}\gamma_i=1$.

Consider the set $K$ of all ordered triples 
$(\gamma,\alpha,\beta)\in[0,\infty)^9$ satisfying the hypotheses
such that $\max_{1\le i\le 3} \gamma_i=1$. 
Since $K$ is defined by finitely many linear inequalities and equalities, it is closed.
It is compact by the assumption of $\tau$--admissibility of ${\gamma}$.
The function
$\Lambda_d({\gamma}) - \Lambda_d({\alpha}) - \Lambda_d({\beta})$
is continuous on $K$, and vanishes nowhere on $K$ by the preceding lemma.
Therefore 
$\Lambda_d(\alpha)+\Lambda_d(\beta)-\Lambda_d(\gamma)\le -\rho'$ for some constant $\rho'<0$.
Since $\Lambda_d(\gamma)$ is bounded above,
$-\rho'\le -\rho\Lambda_d(\gamma)$ provided that $\rho$ is sufficiently small.
\end{proof}

\section{Symmetrizations}\label{section:symms}

Regard $\reals^d$ as the Cartesian product $\reals^{d-1}\times\reals^1$
with coordinates $x=(x',x_d)$.

\begin{definition} 
Let $E\subset\reals^d$ be a Lebesgue measurable set.
\begin{enumerate}
\item
If $|E|>0$,
$E^\Star$ denotes the closed ball in $\reals^d$,
centered at $0$, that satisfies $|E^\Star|=|E|$;
$E^\Star=\reals^d$ if $|E|=\infty$. 
If $|E|=0$ , $E=\emptyset$.

\item
$E^\dagger$ denotes the Schwarz symmetrization of $E$
with respect to the first $d-1$ coordinates. 
For $t\in\reals$ define $r(t)$ by
$\omega_{d-1} r(t)^{d-1} = |\{y\in\reals^{d-1}: (y,t)\in E\}|$.
Then
$E^\dagger$ is the set of all $(x',x_d)\in\reals^{d-1}\times\reals^1$
such that $r(x_d)>0$ and $|x'|\le r(x_d)$.
For an ordered triple $\bE=(E_j: 1\le j\le 3)$
we define $bE^\star = (E_j^\star: 1\le j\le 3)$
and analogously define $bE^\dagger$ and $bE^\sstar$.

\item
$E^\star$ denotes the Steiner symmetrization of $E$
with respect to the last coordinate. That is, $E^\star$ is
the set of all $(x',t)\in\reals^{d-1}\times \reals^1$ 
such that $|E^{x'}|>0$ and $|t|\le \tfrac12 |E^{x'}|$,
where $|E^{x'}|$ is the one-dimensional Lebesgue measure of
$\{t\in\reals: (x',t)\in E\}$.

\item
$E^{\sstar} = (E^\dagger)^\star$.
\end{enumerate}
\end{definition}

Each of these symmetrizations is defined to be empty for any set of Lebesgue measure zero.
In general, they satisfy
\begin{equation} |E^\Star|=|E^\star|=|E^\dagger|=|E^\sstar|=|E|.  \end{equation}

Schwarz symmetrization has a key monotonicity property:
For any set $E\subset\reals^d$, the function of $y\in\reals^{d-1}$ defined to be
$|(E^\dagger)^y|$ is a function of $|y|$ alone, and moreover is a nonincreasing function.
Likewise, for Steiner symmetrization,
$\reals\owns t\mapsto |\{y: (y,t)\in E^\star \}|$ is an even function
which is nonincreasing on $[0,\infty)$.

Define \begin{equation} |{\EE}|  = \max_{1\le j\le 3} |E_j|. \end{equation}

\begin{lemma}
For any Lebesgue measurable sets $E_j\subset\reals^d$ with finite measures,
\begin{align}
& \scriptt(\mbf{E}) \le \scriptt(\mbf{E}^\star) \label{eq:steinersymm}
\\& \scriptt(\mbf{E}) \le \scriptt(\mbf{E}^\dagger) \label{eq:schwarzsymm}
\\& \scriptt(\mbf{E}) \le \scriptt(\mbf{E}^\sstar). \label{eq:doublesymm}
\end{align}
\end{lemma}

\begin{proof}
Inequality \eqref{eq:steinersymm} is obtained by applying the Riesz-Sobolev inequality
to parallel one-dimensional slices of $\reals^d$ in a well-known manner; see for instance
\cite{burchard}.  Inequality \eqref{eq:schwarzsymm} is 
obtained in the same way by working with parallel $d-1$--dimensional slices.
The final inequality is obtained by applying first \eqref{eq:steinersymm}, then \eqref{eq:schwarzsymm}.
\end{proof}

An immediate consequence of the preceding lemma is:
\begin{corollary}
If $\mbf{E}$ is a $\delta$--near extremizer of the Riesz-Sobolev inequality then
$\mbf{E}^\star$, $\mbf{E}^\dagger$, and $\mbf{E}^\sstar$ are also $\delta$--near extremizers.
\end{corollary}

We often regard two measurable sets as identical if their symmetric difference is a Lebesgue null set
of the appropriate dimension.
Let $E\subset\reals^d$ be any Lebesgue measurable set with $0<|E|<\infty$.
$|E\symdif E^\star|=0$ if and only if for almost every $y\in\reals^{d-1}$,
$E^y$ is a Lebesgue null set or differs from an interval centered at $0\in\reals^1$ 
by a one-dimensional Lebesgue null set.
Likewise,
$|E\symdif E^\dagger|=0$ if and only if for almost every $t\in\reals^{1}$,
$\{y\in\reals^{d-1}: (y,t)\in E\}$ is a $d-1$--dimensional Lebesgue null set,
or differs from  ball centered at $0\in\reals^{d-1}$ by a $d-1$--dimensional Lebesgue null set.
A consequence is that
$(E^\star)^\star=E^\star$ and $(E^\dagger)^\dagger = E^\dagger$; that is, 
their respective symmetric differences are to $d$--dimensional Lebesgue null sets.
More is true:

\begin{lemma} \label{lemma:idempotent}
For any Lebesgue measurable set $E\subset\reals^d$ with $0<|E|<\infty$,
\begin{equation} (E^\sstar)^\dagger = E^\sstar 
\ \text{ and }\   E^\sstar = (E^\sstar)^\sstar.  \end{equation} \end{lemma}

\begin{proof}
By its definition and by virtue of the monotonicity property of Schwarz symmetrization, 
$E^\sstar=\{(y;t): |t|\le \tfrac12 \phi(|y|)\}$ (up to a Lebesgue null set)
where $\phi:[0,\infty)\to [0,\infty]$ is nonincreasing.
From this it is apparent that for each $t$,
$\{y\in\reals^{d-1}: (y;t)\in E^\sstar\}$ is a ball 
centered at $0\in\reals^{d-1}$, or is the empty set, or is all of $\reals^{d-1}$.
Consequently 
$(E^{\sstar})^\dagger = E^\sstar$. 
The second conclusion follows since $(A^\sstar)^\sstar=A^\sstar$ for any set $A$.
\end{proof}

In the analysis below,
we seek to extract information about ${\EE}$
from the hypothesis that ${\EE}$ is a near-extremizer of the Riesz-Sobolev inequality.
One of the leading ideas is that partial symmetrizations  such as
${\EE}^\star$ and ${\EE}^\sstar$ 
enjoy enhanced regularity which make it easier to extract
information from their status as near-extremizers;
yet they also depend on ${\EE}$ in such a way that conclusions about them provide 
useful information about ${\EE}$ itself.
In contrast, the full symmetrization $E^\Star$ retains
no information about $E$, except for the value of $|E|$.
Given $\EE$, we will first study $\EE^\sstar$,
then will use information gleaned about $\EE^\sstar$
to gain information about $\EE^\dagger$,
and finally will use this information to study $\EE$.

\section{Structure of doubly symmetric near-extremizers} 
Recall that $\EE^\sstar$ is obtained from $\EE=(E_1,E_2,E_3)$
by applying first Schwarz symmetrization, then Steiner symmetrization, to each of the three sets $E_j$.
$\mbf{A} = \EE^\sstar$ is doubly symmetric; it satisfies both $A^\dagger=A$
and $A^\star=A$.

\begin{proposition} \label{prop:symmetrizedcompactness}
Let $d\ge 2$ and $\tau>0$.
For any $\eps>0$ there exists $\delta>0$ with the following property.
Let ${\mathbf E}$ be any 
$\tau$--admissible ordered triple of Lebesgue measurable subsets of $\reals^d$
that is a $(1-\delta)$--near extremizer of the Riesz-Sobolev inequality 
and satisfies 
$\mbf{E}=\mbf{E}^\sstar$.
There exist an ellipsoid $\scripte\subset\reals^d$
and $\br\in[0,\infty)^3$ such that
\begin{equation} \big|\, E_j\symdif r_j \scripte \,\big|<\eps |{\mathbf E}|
\ \text{ for all } j\in\{1,2,3\}.  \end{equation}
\end{proposition}

A consequence is that $\br$ is $\tau-O(\eps)$--admissible.
An equivalent formulation is that if 
$\mbf{E}$ is any $\tau$--admissible $\delta$--near extremizer,
not necesarily satisfying any symmetry hypothesis,
then there exists an ellipsoid such that
$\big|\, E_j^\sstar\symdif r_j \scripte \,\big|<\eps |{\mathbf E}|$.

For $x'\in\reals^{d-1}$ define $h_j(x')=|\{t\in\reals: (x',t)\in E_j\}| = |E_j^{x'}|$. 
Then since $E_j=E_j^\dagger$,
$h_j(x')$ depends only on $|x'|$, and is a nonincreasing function
of $|x'|$. Since $E_j=E_j^\star$, 
$E=\{(x',t): |t|\le \tfrac12 h_j(x')\}$ up to a Lebesgue null set.

For $k\in\integers$
define the sets $E'_{j,k}\subset\reals^{d-1}$ by
\begin{equation}
E'_{j,k}=\{x'\in\reals^{d-1}: 2^k\le h_j(x') < 2^{k+1} \}
\end{equation}
and  define
\begin{equation}
E_{j,k}=E_j\cap \pi^{-1}(E'_{j,k})=\{x\in E_j: \pi(x) \in E'_{j,k}\}
\end{equation}
where $\pi:\reals^d\to\reals^{d-1}$ is the projection $\pi(x',t)=x'$.

\begin{lemma} \label{lemma:theta}
There exist constants $c,C\in\reals^+$,
depending only on the dimension $d$, such that
\begin{equation}
\scriptt(E_{1,k_1},E_{2,k_2},E_{3,k_3})
\le C \theta
(E_{1,k_1},E_{2,k_2},E_{3,k_3})
\prod_{j=1}^3|E_{j,k_j}|^{2/3}
\end{equation} 
where
\begin{equation} \theta(E_{1,k_1},E_{2,k_2},E_{3,k_3})
= \min_{m,n}2^{-|k_m-k_n|/3}
\cdot \min_{\mu,\nu} \big(|E'_{\mu,k_\mu}|\,/\,|E'_{\nu,k_\nu}|\big)^{1/3}
\end{equation} 
where the minima are taken over all $m\ne n\in\{1,2,3\}$
and over all $\mu\ne\nu\in\{1,2,3\}$, respectively.
\end{lemma}

By its definition, this quantity $\theta$ does not exceed $1$,
but it may be smaller. We aim to exploit these potential small values.

\begin{proof}
In proving the lemma we will use the trilinear forms $\scriptt$
in $\reals^\kappa$ with $\kappa$ equal to each of $d$, $d-1$, and $1$, denoting each by
$\scriptt_\kappa$ to indicate which dimension is in play. 
Consider any Lebesgue measurable sets $A_j\subset\reals^{\kappa}$ with finite measures.
Write $\{1,2,3\} = \{k,m,n\}$ where $|A_m|\le |A_k|\le |A_n|$.
Then 
\begin{align*} 
\scriptt_\kappa(A_1,A_2,A_3) &\le |A_m|\cdot|A_k|
\\ & \le |A_m|^{1/3} |A_k|^{1/3}|A_n|^{-2/3} \cdot\prod_{j=1}^3 |A_j|^{2/3}
\\ &\le |A_m|^{1/3} |A_n|^{-1/3} \cdot\prod_{j=1}^3 |A_j|^{2/3}.  \end{align*}

Let $m\ne n$ and $\mu\ne\nu$.
\begin{align*}
\scriptt_d(E_{1,k_1},E_{2,k_2},E_{3,k_3})
& = 
\int
\scriptt_{1}(E_{1,k_1}^{x'_1},E_{2,k_2}^{x'_2},E_{3,k_3}^{x'_3})\,d\lambda
(x'_1,x'_2,x'_3)
\\& \le
\int
C2^{-|k_m-k_n|/3} \prod_{j=1}^3 |E_{j,k_j}^{x'_j}|^{2/3}\one_{E'_{j,k_j}}(x'_j)
\,d\lambda
(x'_1,x'_2,x'_3)
\intertext{by the above bound for $\scriptt_\kappa(A_1,A_2,A_3)$ with 
$\kappa=1$ and $A_j = E^{x'_j}_{j,k_j}$}
\\& =
C2^{-|k_m-k_n|/3} 
\prod_{i=1}^3  2^{2k_i/3}
\scriptt_{d-1}(E'_{1,k_1},E'_{2,k_2},E'_{3,k_3})
\\& \le C2^{-|k_m-k_n|/3} 
\prod_{i=1}^3 2^{2k_i/3}
|E'_{\mu,k_\mu}|^{1/3}|E'_{\nu,k_\nu}|^{-1/3}
\prod_{j=1}^3 |E'_{j,k_j}|^{2/3}
\\&  \le
C 2^{-|k_m-k_n|/3} 
|E'_{\mu,k_\mu}|^{1/3}|E'_{\nu,k_\nu}|^{-1/3}
\prod_{j=1}^3 |E_{j,k_j}|^{2/3}
\end{align*}
where the value of the constant $C$ may change from each occurrence to the next.
\end{proof}

\begin{lemma}
For any $d\ge 2$ and $\tau>0$
there exist constants $c,C\in\reals^+$ 
such that for any $\tau$--admissible ordered triple $\mbf{E}$ of subsets of $\reals^d$
satisfying 
$\scriptt(\mbf{E})\ge \tfrac12 \scriptt(\mbf{E}^\Star)$
there exist $K_j\in\integers$
such that 
\begin{align}
|K_i-K_j| &\le C\ \text{ for all $i\ne j\in\{1,2,3\}$} 
\\ |E_{j,K_j}| &\ge c|\mbf{E}|\ \text{ for all $j\in\{1,2,3\}$.}
\end{align} 
\end{lemma}

\begin{proof}
If $C$ is sufficiently large, and if $c$ is sufficiently small,
and if no such triple $(K_i: 1\le i\le 3)$ exists,
then Lemma~\ref{lemma:theta}
implies that $\scriptt(\mbf{E}) \ll |\mbf{E}|^2$.
On the other hand, $\scriptt(\EE^\Star)$
is comparable to $|\mbf{E}^\Star|^2$ so long as $\tau$ is bounded below.
Thus the hypothesis $\scriptt(\mbf{E})\ge \tfrac12 \scriptt(\mbf{E}^\Star)$
is contradicted.
\end{proof}

\begin{definition}
A special dilation of $\reals^d$ is a linear transformation of the form
\begin{equation} T(x_1,\dots,x_d)  = (rx_1,\dots,rx_{d-1},\rho x_d)\end{equation}
for some $r,\rho\in\reals^+$.
\end{definition}

The next result is a direct application of the preceding lemma.
\begin{corollary}
Let $d\ge 2$.
There exist $C,c\in\reals^+$ such that under the hypotheses of the preceding lemma,
there exists a special dilation $T$ of $\reals^d$ with determinant equal to $1$
such that after replacement of ${\mathbf E}$ by $(T(E_1),T(E_2),T(E_3))$, 
\begin{equation}
\sum_{|k|\le C} |E_{j,k}| \ge c|{\mathbf E}|
\text{ for all $j\in\{1,2,3\}$.}
\end{equation}
\end{corollary}

\begin{lemma}
Let $d\ge 2$, $\tau>0$, and $C_0,c_0\in\reals^+$.
For any $\eps>0$ there exist $\delta>0$ and $A<\infty$ with the following property.
Let $\mbf{E}$ be a $\tau$--admissible ordered triple of Lebesgue measurable subsets of $\reals^d$
satisfying 
$\scriptt(\mbf{E})\ge (1-\delta)\scriptt(\mbf{E}^\Star)$.
Suppose that 
\begin{equation*}
\sum_{|k|\le C_0} |E_{j,k}| \ge c_0|\mbf{E}|
\text{ for all $j\in\{1,2,3\}$.}
\end{equation*}
Then 
\begin{equation}\label{nogap}
\sum_{|k|\ge A} |E_{j,k}| \le \eps|\mbf{E}|.
\end{equation}
\end{lemma}

\begin{proof}
We argue by contradiction.
Let $d,\tau,C_0,c_0$ be given.

If the lemma were not true then one of two possibilities must hold. 
In the first case, there exists
$\eps>0$  
such that for any $\delta>0$, $\sigma>0$, and $A\in\integers^+$ 
there exists 
a $\tau$--admissible ordered triple $\mbf{E}$ satisfying 
$\scriptt(\mbf{E})\ge (1-\delta)\scriptt(\mbf{E}^\Star)$
such that
\begin{align}
|\cup_{k \ge 2A} E_{m,k}| &\ge \eps|\mbf{E}|\ \text{ for some $m\in\{1,2,3\}$}
\\ |\cup_{A< k< 2A} E_{j,k}| &\le \sigma|\mbf{E}|\ \text{ for all $j\in\{1,2,3\}$}
\\ |\cup_{k\le A } E_{j,k}| &\ge c|\mbf{E}|\ \text{ for all $j\in\{1,2,3\}$.}
\end{align}
Since the form $\scriptt(E_1,E_2,E_3)$ is invariant under permutation of its arguments,
we may assume that the first inequality holds for $m=1$.

The second case is much like the first, except that the sign of the index $k$ in the
above inequalities is in effect reversed:
\begin{align}
|\cup_{k \le -2A} E_{1,k}| &\ge \eps|\mbf{E}|
\\ |\cup_{-2A<k< -A} E_{j,k}| &\le \sigma|\mbf{E}|\ \text{ for all $j\in\{1,2,3\}$}
\\ |\cup_{k\ge -A } E_{j,k}| &\ge c|\mbf{E}|\ \text{ for all $j\in\{1,2,3\}$.}
\end{align}
We will discuss only the first case; the same reasoning with $k$ replaced by $-k$ will apply 
equally well to the second. 

Decompose $E_1=E_1^+\cup E_1^-\cup E_1^0$
by setting 
\begin{align*}
E_1^+ =\bigcup_{k\ge 2A} E_{1,k},
\qquad E_1^- =\bigcup_{k\le A} E_{1,k},
\qquad E_1^0 =\bigcup_{A<k< 2A} E_{1,k}.
\end{align*}
For $j\in\{2,3\}$ decompose $E_j=E_j^+\cup E_j^-$ by setting
\begin{align*}
E_j^+ =\bigcup_{k\ge \tfrac32 A} E_{j,k},
\qquad E_j^- =\bigcup_{k< \tfrac32 A} E_{j,k}.
\end{align*}

By inserting these three decompositions into $\scriptt(E_1,E_2,E_3)$ and invoking
multilinearity of $\scriptt$, we obtain a decomposition
\[ \scriptt(E_1,E_2,E_3)
= \scriptt(E_1^+,E_2^+,E_3^+)
+ \scriptt(E_1^-,E_2^-,E_3^-)
+ \scriptt(E_1^0,E_2,E_3)
\text{ plus six more terms.}\]
All four terms of the form $\scriptt(E^0,E_2^\pm,E_3^\pm)$ have been combined into one single term,
so that the total number of terms is nine, rather than twelve.

Each of the six terms not shown explicitly takes the form $\scriptt(E_1^\pm,E_2^\pm,E_3^\pm)$,
where the three $\pm$ signs are not all equal. By Lemma~\ref{lemma:theta},
\[ \scriptt(E_1^\pm,E_2^\pm,E_3^\pm) \le C2^{-cA}|\mbf{E}|^2\]
for each of those six terms.
Moreover
\[ \scriptt(E_1^0,E_2,E_3)\le |E_1^0|^{2/3}|E_2|^{2/3}|E_3|^{2/3} \le C \sigma^{2/3}|\mbf{E}|^2,\]
where the constant $C$, like other constants in this argument, depends on $\tau$.
Thus
\begin{align*} \scriptt(E_1,E_2,E_3)
&\le  \scriptt(E_1^+,E_2^+,E_3^+) + \scriptt(E_1^-,E_2^-,E_3^-) + (C2^{-cA}+C\sigma^{2/3}) |\mbf{E}|^2.
\\& \le  \scriptt(E_1^+,E_2^+,E_3^+) + \scriptt(E_1^-,E_2^-,E_3^-) 
+ C(2^{-cA}+\sigma^{2/3})\scriptt(\mbf{E}^\Star).
\end{align*}

By the Riesz-Sobolev inequality followed by Corollary~\ref{cor:Lambdastrongtriangle},
\begin{align*}
\scriptt(E_1^+,E_2^+,E_3^+) + \scriptt(E_1^-,E_2^-,E_3^-) 
&\le \Lambda_d(|E_1^+|,|E_2^+|,|E_3^+|) + \Lambda_d(|E_1^-|,|E_2^-|,|E_3^-|)
\\& \le (1-\rho) \Lambda_d(|E_1^+|+|E_1^-|,|E_2^+|+|E_2^-|,|E_3^+|+|E_3^-|) 
\\& = (1-\rho)\Lambda_d(|E_1\setminus E_1^0|,|E_2|,|E_3|) 
\\& \le  (1-\rho)\Lambda_d(|E_1|,|E_2|,|E_3|) 
\\& =  (1-\rho)\scriptt(\mbf{E}^\Star)
\end{align*}
where $\rho>0$ depends only  on $d, \eps, \tau$.
In all,
\begin{equation}
\scriptt(\EE)
\le  \big[(1-\rho) + C2^{-cA}+\sigma^{2/3}\big]\scriptt(\mbf{E}^\Star).
\end{equation}
If $A$ is sufficiently large and $\sigma,\delta$ are sufficiently small, this contradicts the assumption
that $\scriptt(\EE) \ge (1-\delta) \scriptt(\mbf{E}^\Star)$.
\end{proof}

We say that a family of Lebesgue measurable subsets of $\reals^d$ is precompact
if every sequence
$(E_\nu)_{\nu\in\naturals}$ of sets
in this family has a subsequence that converges to some set in the sense that
$\lim_{k\to\infty} |E_{\nu_k}\symdif E|=0$.
This is equivalent to precompactness of the associated family of indicator
functions $\one_E$ in $L^1(\reals^d)$.

In the next lemma,
$E^\nu=\{(x',t)\in\reals^{d-1}\times\reals^1: |t|\le \tfrac12 h_\nu(x')\}$
where $h_\nu$ is a radial nonincreasing function,
and $E^\nu_k=\{(x',t)\in E^\nu: 2^{k-1}\le h_\nu(x')<2^{k}\}$.

\begin{lemma} \label{lemma:sstarprecompactness}
Let $\varphi:\naturals\to(0,\infty)$ satisfy $\lim_{|k|\to\infty} \varphi(k)=0$.
Let $(\varrho_\nu: \nu\in\naturals)$ satisfy $\lim_{\nu\to\infty}\varrho_\nu=0$.
Let $(E^\nu)_{\nu\in\naturals}$ be a sequence of subsets of $\reals^d$
satisfying $E^\nu = (E^\nu)^\sstar$ and $|E^\nu|=|\bb_d|$.
Suppose that for all $\nu$ and all $K\in\naturals$,
\begin{equation} \sum_{|k|>K} |E^\nu_k| \le \varphi(K)+\varrho_\nu.  \end{equation}
Then $\{E^\nu\}$ is a precompact family of subsets of $\reals^d$.
\end{lemma}

The straightforward and elementary proof is omitted. \qed


\begin{lemma}
Let $d\ge 2$ and $\tau>0$.
There exist functions $\varrho$ and $\varphi$, depending only on $d,\tau$, satisfying
$\lim_{\delta\to 0}\varrho(\delta) = 0$ and
$\lim_{k\to \infty} \varphi(k) = 0$, with the following property.
For any $\tau$--admissible ordered triple $\mbf{E}$ of subsets of $\reals^d$ satisfying
$\mbf{E}=\mbf{E}^\sstar$ and
$\scriptt(\mbf{E}) \ge (1-\delta)\scriptt(\mbf{E}^\Star)$,
there exists a special dilation $T$ of $\reals^d$ such that $\max_j |T(E_j)|=1$ and
\begin{equation} \sum_{|k|>K} |(TE_j)_{k}| \le (\varphi(K)+\varrho(\delta))|\mbf{T}(\mbf{E})| \end{equation}
for all $K\in\naturals$ and $j\in\{1,2,3\}$.
\end{lemma}

This is simply a reformulation of what has been shown above. \qed

\begin{lemma} \label{lemma:getellipsoid}
Let $d\ge 2$ and $\tau>0$.
Given any $\eps>0$ there exists $\delta>0$ such that
for any $\tau$--admissible ordered triple $\mbf{E}$
satisfying $\scriptt(\mbf{E}) \ge (1-\delta)\scriptt(\mbf{E}^\Star)$
and $\mbf{E} = \mbf{E}^\sstar$,
there exists an ellipsoid $\scripte$ such that
$|E_1\symdif \scripte|<\eps |\mbf{E}|$.
\end{lemma}

\begin{proof}
By the preceding lemma and the dilation-invariance of the hypotheses and conclusion,
it suffices to prove this under the additional assumption that
$\sum_{|k|>K} |E_{j,k}| \le (\varphi(K)+\varrho(\delta))|\mbf{E}|$ 
for all $k\in\integers$ and $j\in\{1,2,3\}$,
where $\varphi(k),\varrho(\delta)\to 0$ as $|k|\to\infty$
and as $\delta\to 0$, respectively.

Suppose that the lemma were false. Then there would exist $\eps>0$ and a sequence 
$(\mbf{E^\nu}: \nu\in\naturals)$ of ordered triples of subsets of $\reals^d$ satisfying all of
the above hypotheses with parameters $\delta_\nu$ tending to zero, such that 
\begin{equation} \label{tobecontradicted}
\liminf_{\nu\to\infty}\,\inf_\scripte\, |\scripte\symdif E^\nu_1|>0, \end{equation}
where the infimum is taken over all ellipsoids $\scripte\subset\reals^d$.
By Lemma~\ref{lemma:sstarprecompactness} and a diagonal argument, 
there exist a subsequence $(\mbf{E^{\nu_n}})$
and an ordered triple $\mbf{E}$ of subsets of $\reals^d$
such that $\lim_{n\to\infty}|E^{\nu_n}_j\symdif E_j| = 0$ for $j\in\{1,2,3\}$.
Since $|E^{\nu_n}_j|\to |E_j|$ for all $j$,
the limiting triple $\mbf{E}$ is also $\tau$--admissible.
Moreover
$\scriptt(\mbf{E}) = \lim_{n\to\infty} \scriptt(\mbf{E^{\nu_n}})$
and 
$\scriptt(\mbf{E}^\Star) = \lim_{n\to\infty} \scriptt((\mbf{E^{\nu_n}})^\Star)$.
Since $\lim_{\nu\to\infty}\delta_{\nu} = 0$, 
we conclude that
$\scriptt(\mbf{E}) = \scriptt(\mbf{E}^\Star)$.  
Since $\mbf{E}$ is strictly admissible,
Burchard's theorem \cite{burchard} guarantees that each set $E_j$ is an ellipsoid.
This contradicts \eqref{tobecontradicted}.
\end{proof}

Burchard's theorem yields supplementary conclusions which will be exploited below: 
The three ellipsoids $E_j$
are homothetic, and their centers $c_j$ satisfy $\sum_{j=1}^3 c_j=0$. 

We are now ready to complete the proof of 
Proof of Proposition~\ref{prop:symmetrizedcompactness}
concerning near-extremizers of the Riesz-Sobolev inequality
that enjoy the symmetry $\mbf{E}=\mbf{E}^\sstar$.

\begin{proof}[Proof of Proposition~\ref{prop:symmetrizedcompactness}]
Let $d\ge 2$ and $\tau>0$.  Let $\mbf{E}$ be $\tau$--admissible, 
and satisfy $\scriptt(\mbf{E})\ge (1-\delta)\scriptt(\mbf{E}^\Star)$ and $\mbf{E}=\mbf{E}^\sstar$.
By Lemma~\ref{lemma:getellipsoid},
there exists a $\tau$--admissible ordered triple $\scripte = (\scripte_1,\scripte_2,\scripte_3)$
of ellipsoids in $\reals^d$ such that for each $j\in\{1,2,3\}$,
$|\scripte_j\symdif E_j|<\eps|\mbf{E}|$ where $\eps=o_\delta(1)$.  Then
\[ \scriptt(\scripte) = \scriptt(\mbf{E}) + O(\eps|\mbf{E}|^2)\]
and $|\scripte_j|=|E_j|+O(\eps|\mbf{E}|)$.
It follows that from the $\tau$--admissibility of $\EE$ and the smallness of these
symmetric differences that 
$\scriptt(\scripte) \ge (1-C\eps)|\scripte|$ and 
$(\scripte_1,\scripte_2,\scripte_3)$ is $\tau-C\eps$--admissible.

Arguing by contradiction as in the proof of the preceding lemma,
but this time using the two supplementary conclusions of Burchard's theorem, we conclude
that there exist a single  ellipsoid $\tilde\scripte\subset\reals^d$, 
elements $y_j\in\reals^d$ satisfying $\sum_{j=1}^3 y_j=0$, and a $\tau-o_\delta(1)$--admissible
ordered triple $\br\in(0,\infty)^3$
such that \[ |E_j \symdif (r_j\tilde \scripte +y_j)| \le \eps \text{ for $j\in\{1,2,3\}$}\]
where $\eps=o_\delta(1)$ provided that $\tau>0$ remains fixed.

Because $E_j = E_j^\sstar$, the same holds with $y_j=0$ for all $j$.
\end{proof}

\section{Structure of Schwarz--symmetrized near-extremizers}

We have characterized near-extremizers with the symmetry property $\EE = \EE^\sstar$.
Next we use that characterization to analyze near-extremizers with the less restrictive
symmetry property $\EE=\EE^\dagger$.

By a vertical skew-shift $\mbf{T}$
we mean an element $\mbf{T}=(T_1,T_2,T_3)\in\baff(d)$ 
of the form $T_j(x',x_d) = (x',x_d+\ell_j(x'))$
where $\ell_j:\reals^{d-1}\to\reals$ is an affine mapping. 
Part of the definition of $\baff(d)$
is the requirement that $\sum_{j=1}^3 \ell_j\equiv 0$.

\begin{proposition} \label{prop:lesssymmetrizedcompactness}
Let $d\ge 2$.
Let $\tau>0$.
There exists $c>0$ such that
for any $\eps>0$
there exists $\delta>0$ with the following property.
Let ${\mathbf E}=(E_j: 1\le j\le 3)$
be any 
$\tau$--admissible 
ordered triple of Lebesgue measurable subsets of $\reals^d$
that is a $(1-\delta)$--near extremizer of the Riesz-Sobolev inequality
and satisfies $\mbf{E}^\dagger=\mbf{E}$.
There exists a vertical skew-shift $\mbf{T}\in\baff(d)$
such that for each $j\in\{1,2,3\}$,
\begin{equation} \big|\, T_j(E_j)\symdif r_j \bb_{d} \,\big|<\eps |{\mathbf E}| 
\ \text{ for all } j\in\{1,2,3\}, \end{equation}
where $r_j=|E_j|^{1/d}/|\bb_d|^{1/d}$.
\end{proposition}

The conclusion is nearly the same as that of
Proposition~\ref{prop:symmetrizedcompactness}; the significant change
is the weakening of the hypothesis
from $\mbf{E}^\dagger{}^\star=\mbf{E}$ to $\mbf{E}^\dagger=\mbf{E}$.

In the proof, $o_\delta(1)$ denotes any quantity that depends on $d,\tau,\delta$ and
that tends to zero as $\delta$ tends to zero while $d,\tau$ remain fixed.
We begin by applying 
Proposition~\ref{prop:symmetrizedcompactness} to $\mbf{A}=\EE^\sstar = \EE^\star$, 
which by Lemma~\ref{lemma:idempotent} satisfies $(\mbf{A})^\sstar=\mbf{A}$. 
By making a linear change of variables in $\reals^d$ of the form
$(x',x_d)\mapsto (rx',\rho x_d)$ for appropriate $r,\rho\in\reals^+$
we may assume that $\br=(r_j)_{1\le j\le 3}=(\omega_d^{-1/d} |E_j|^{1/d})_{1\le j\le 3}$ 
is $\tau$--admissible, that $\max_{1\le j\le 3} r_j =1$,
and that \[\big|\,E_j^\star\symdif r_j \bb_d\big|  = o_\delta(1)\]  for each $j\in\{1,2,3\}$.
Therefore
\begin{equation} \big|\,|E_j^{r_j x'}| - \omega_{d-1}r_j^{d-1} (1-|x'|^2)^{(d-1)/2}\,\big| =o_\delta(1)
\end{equation} for all $x'\in\bb_{d-1}$ outside a set of measure $o_\delta(1)$.

Let $\delta'>0$.
By Lemma~\ref{lemma:zeroprime}, if $\delta$ is sufficiently small then
for all $x'_1\in\bb_{d-1}$ outside a set of measure $<\delta'$, 
for all $(x'_2,x'_3)\in\bb_{d-1}^2$ satisfying $\sum_{j=1}^3 r_j x'_j=0$
outside a set of $\lambda^{x_1}_\br$ measure $<\delta'$,
the ordered triple
$(|E_1^{r_1 x'_1}|, |E_2^{r_2 x'_2}|, |E_3^{r_3 x'_3}|)$
is $\tau'$--admissible, where $\tau'>0$ depends on $\delta',\tau,d$ 
but is independent of $\delta$ so long as $\delta$ is sufficiently
small as a function of $\delta',\tau,d$.
The same statements hold with the roles of the indices $1,2,3$ permuted arbitrarily.

We denote by $\scriptt_1$ the expression
$\scriptt_1(A_1,A_2,A_3) = \int_{(\reals^1)^3} \prod_{j=1}^3 \one_{A_j}(x_j)
\,d\lambda_\br(x_1,x_2,x_3)$ acting on subsets of $\reals^1$,
and by $\scriptt$ the corresponding expression for subsets of $\reals^d$.
These are related by
\begin{align*}
\scriptt(E_1,E_2,E_3)
&= 
\int_{(\reals^{d-1})^3} 
\scriptt_1
\big( E_1^{r_1 x'_1}, E_2^{r_2 x'_2}, E_3^{r_3 x'_3} \big)
\,d\lambda_\br(x'_1,x'_2,x'_3)
\\ &\le o_\delta(1) + \int_{\bb_{d-1}^3} 
\scriptt_1
\big( E_1^{r_1 x'_1}, E_2^{r_2 x'_2}, E_3^{r_3 x'_3} \big)
\,d\lambda_\br(x'_1,x'_2,x'_3)
\end{align*}
where the term $o_\delta(1)$ majorizes the contribution
of points in $(\reals^{d-1})^3\setminus \bb_{d-1}^3$.
Therefore
\begin{align*}
\scriptt(E_1,E_2,E_3)
&\le o_\delta(1) + 
\int_{\bb_{d-1}^3} 
\scriptt_1
\big( (E_1^{r_1 x'_1})^\star, (E_2^{r_2 x'_2})^\star, (E_3^{r_3 x'_3})^\star \big)
\,d\lambda_\br(x'_1,x'_2,x'_3)
\\& \le o_\delta(1) + 
\scriptt(\mbf{E}^\star).
\end{align*}

Therefore
\begin{equation}
\int 
\Big(\scriptt_1
\big( (E_1^{r_1 x'_1})^\star, (E_2^{r_2 x'_2})^\star, (E_3^{r_3 x'_3})^\star \big)
- \scriptt_1
\big( E_1^{r_1 x'_1}, E_2^{r_2 x'_2}, E_3^{r_3 x'_3} \big)
\Big)
\,d\lambda_\br(x'_1,x'_2,x'_3)
= o_\delta(1).
\end{equation}
By the Riesz-Sobolev inequality, the integrand is nonnegative.
Therefore by Chebyshev's inequality,
\begin{equation}\label{eq:fromCheby}
\scriptt_1
\big( (E_1^{r_1 x'_1})^\star, (E_2^{r_2 x'_2})^\star, (E_3^{r_3 x'_3})^\star \big)
- \scriptt_1
\big( E_1^{r_1 x'_1}, E_2^{r_2 x'_2}, E_3^{r_3 x'_3} \big)
= o_\delta(1)
\end{equation}
for all $(x'_1,x'_2,x'_3)\in\bb_{d-1}^3$ satisfying
$\sum_{j=1}^3 r_j x'_j=0$
except for a set whose $\lambda_\br$ measure is $o_\delta(1)$.

Therefore for all $x'_1\in \bb_{d-1}$ outside a set of measure $o_\delta(1)$,
the $\lambda^{x'_1}_{\br}$ measure of the set of all $(x'_2,x'_3)\in \bb_{d-1}\times\bb_{d-1}$
such that $(E_1^{r_1 x'_1},E_2^{r_2 x'_2},E_3^{r_3x'_3})$
is $\tau'$--admissible and \eqref{eq:fromCheby} holds
is bounded below by a positive constant that depends on $d,\tau$ but
is independent of $\delta$ provided that $\delta$ is sufficiently small.

The one-dimensional case of our main theorem, proved in \cite{christRS3},
states that
if $d,\tau$ are fixed and $\delta$ is sufficiently small then
for any such $x'_1$ there exists an interval $J_1^{x'_1}\subset\reals^1$ satisfying
\begin{equation}
\label{setupforfunctleqn}
\big| E_1^{r_1 x'_1} \symdif J^{x'_1}_1\big| \le o_\delta(1)|E_1^{r_1 x'_1}|.
\end{equation}
The corresponding conclusion holds for the sets $E_j$ for $j=2,3$ with corresponding
intervals $J_j^{x'_j}$.
A companion conclusion established in \cite{christRS3} is that the centers 
$c_j(x'_j)$ of these intervals (which are well-defined for all $x'_j\in\bb_{d-1}$ 
outside the exceptional sets introduced above) satisfy
\begin{equation} 
\label{setupforfunctleqn2}
\sum_{j=1}^3 r_j c_j(x'_j)=o_\delta(1) \end{equation}
whenever $(x'_j: 1\le j\le 3)\in\bb_{d-1}^3$
except for a set of values of $(x'_j: 1\le j\le 3)\in\bb_{d-1}^3$
whose $\lambda_\br$ measure is $o_\delta(1)$.

By Lemma~8.3 of \cite{christyoungest}, \eqref{setupforfunctleqn2} implies the existence of
affine mappings $\ell_j:\reals^{d-1}\to\reals^1$
satisfying $\sum_{j=1}^3 \ell_j\equiv 0$
such that for each $j\in\{1,2,3\}$,
\begin{equation} c_j(r_j x') = \ell_j(x') + o_\delta(1) \end{equation}
for all $x'\in\bb_{d-1}$ outside a set of measure $o_\delta(1)$.

Thus $E_j$ has small symmetric difference with the ellipsoid
\[\scripte_j=\{(x';t): |t-\ell_j(x')|^2<r_j(1-|x'|^2) \}.\] 
Since $E_j$ is invariant under rotations of $\reals^d$ about the $x_d$ axis, this forces 
\begin{equation} |\ell_j(x')| = o_\delta(1)\end{equation} 
for most $x'\in\bb_{d-1}$ and hence uniformly for all $x'\in\bb_{d-1}$
since $\ell_j$ is affine.
Therefore 
\begin{equation} |E_j\symdif  r_j\bb_d|\le o_\delta(1)|\EE|.  \end{equation}
This completes the proof of Proposition~\ref{prop:lesssymmetrizedcompactness}.
\qed

\section{Structure of near-extremizers}

The final stage of the analysis is the removal of the Schwarz symmetry hypothesis $\EE=\EE^\dagger$.

\begin{proposition}\label{prop:finalmainstep}
Let $d\ge 2$.
Let $\tau>0$.
There exists $\tau'>0$ such that
for any $\eps>0$
there exists $\delta>0$ with the following property.
For any $\tau$--admissible ordered triple ${\mathbf E}=(E_j: 1\le j\le 3)$
of Lebesgue measurable subsets of $\reals^d$  
satisfying 
\begin{equation} \scriptt({\mathbf E}) \ge (1-\delta)\scriptt({\mathbf E}^\Star)\end{equation}
there exist $T\in{\baff(d)}$
and a $\tau'$--admissible ordered triple $\br=(r_1,r_2,r_3)$ of positive real numbers such that 
\begin{equation}
\big|\, T_j(E_j) \symdif r_j \bb_d \,\big|<\eps|\EE| \ \text{ for all } j\in\{1,2,3\}.
\end{equation} \end{proposition}

Let $\delta>0$ be small. 
If $\EE$ satisfies the hypotheses then the Schwarz symmetrization $\mbf{E}^\dagger$
is likewise a $\delta$--near extremizer of the Riesz-Sobolev inequality.
Since $(\mbf{E}^\dagger)^\dagger = \mbf{E}^\dagger$,
Proposition~\ref{prop:lesssymmetrizedcompactness} states that
there exist an ellipsoid $\scripte$ of the form
\[\scripte = \{(x'; x_d)\in\reals^{d-1}\times \reals: \alpha^2|x'|^2+\beta^2 x_d^2\le 1\} \]
and $a_j\in\reals$ satisfying $\sum_{j=1}^3 a_j=0$
such that
\begin{equation}\label{eq:schwarzified} 
\big| E_j^\dagger \symdif (r_j\scripte+a_je_d)\big|=o_\delta(1)\end{equation}
for $j\in\{1,2,3\}$,
where $r_j^d|\scripte|=|E_j|$.
By exploiting the action of the group $\baff(d)$, we may reduce to the situation
in which $\scripte = \bb_d$, $a_j=0$ for all $j\in\{1,2,3\}$, and $\max_j r_j=1$
without disturbing the hypotheses on $\EE$.


The conclusion relevant to our purpose contained in \eqref{eq:schwarzified}
is that the Lebesgue measures of the slices $E_j^{(s)}=\{x'\in\reals^{d-1}: (x',s)\in E_j\}$,
which after all are equal to the Lebesgue measures of
the corresponding slices of $E^\dagger$, satisfy
\begin{equation}
\big|\,|E_j^{(r_j t)}|-r_j^{d-1}\omega_{d-1} (1-t^2)^{(d-1)/2}\,\big| =o_\delta(1)
\end{equation}
for all $t\in[-1,1]$ except for a set of measure $o_\delta(1)$, and
\[|E_j\setminus (\reals^{d-1}\times[-r_j,r_j])|=o_\delta(1).\]
According to Lemma~\ref{lemma:zeroprime}, 
for any $\eps>0$ there exist $\tau',\eta>0$ 
such that for any sufficiently small $\delta>0$ there exists 
a partition $[-1,1]=\scriptg\cup\scriptb$
with $|\scriptb|<\eps$
such that for each $t_1\in\scriptg$,
the $\lambda_{t_1}$ measure of the set of all $(t_2,t_3)\in[-1,1]^2$
for which $(|E_j^{(r_jt_j)}|^{1/d-1})_{1\le j\le 3}$
is $\tau'$--admissible is $\ge\eta$.
The same holds with the roles of the three indices $j$ permuted arbitrarily.



Now repeat the proof of Proposition~\ref{prop:lesssymmetrizedcompactness},
in particular the analysis of $\scriptt(E_1,E_2,E_3)$,
with the roles of the two factors $\reals^{d-1}$ and $\reals^1$
in the Cartesian product representation $\reals^{d-1}\times\reals^1$
of $\reals^d$ interchanged. 
Invocation of the case $d=1$ of Theorem~\ref{thm:RShighermain} 
is replaced by an invocation of the $d-1$--dimensional case,
which is valid by induction on the dimension.
Conclude that for each $t\in[-1,1]$ outside a set of measure $o_\delta(1)$, 
there exist an ellipsoid $\scripte_1(t)\subset\reals^{d-1}$ centered at $0$
and a vector $v_1(t)\in\reals^{d-1}$ such that $E_1^{(r_1 t)}$ satisfies
\begin{equation} \label{eq:nearlyloid}
\big |E_1^{(r_1t)} \symdif \big[r_1 \scripte_1(t)+  v_1(t)\big]\big|
= o_\delta(1)|E_1^{(r_1 t)}|.
\end{equation}
Corresponding conclusions hold for the indices $j=2,3$.

The reasoning in the proof of Proposition~\ref{prop:lesssymmetrizedcompactness}
together with the induction-on-dimension hypothesis
also guarantee that the ellipsoids $\scripte_j(t_j)$ and vectors $v_j(t_j)$
are compatible in two respects. Firstly, for all ordered triples
$\bt=(t_1,t_2,t_3)\in[-1,1]^3$ that satisfy $\sum_j r_j t_j=0$
such that $(|E_j^{(r_jt_j)}|^{1/d-1})_{1\le j\le 3}$ is $\tau'$--admissible, 
\begin{equation}\label{eq:displacementscompatible} 
\sum_{j=1}^3 v_j(t_j)=o_\delta(1)\end{equation}
except for $\bt$ in a set of $\lambda_\br$ measure.
Secondly, for each element $\bt$ of this same set of ordered triples, 
there exists an ellipsoid $\scripte(\bt)$
such that the above conclusions hold with 
$\scripte_j(t_j)=\scripte(\bt)$ for all three indices $j$.

In order to complete the proof of Proposition~\ref{prop:finalmainstep}
and hence of Theorem~\ref{thm:RShighermain}, it suffices to show 
that there exists a single ellipsoid $\scripte\subset\reals^{d-1}$, centered at $0$,
such that for all $t\in[-1,1]$ outside a set of Lebesgue measure $o_\delta(1)$,
each ellipsoid $\scripte_j(t)$ is nearly homothetic to $\scripte$ in the sense that
\begin{equation}\label{eq:finalclaim} \big| \scripte_j(t) \symdif (1-t^2)^{1/2}\scripte\big|
= o_\delta(1)|E_j^{(r_j t)}|.  \end{equation}

Since $\br$ is $\tau$--admissible,
there exist $\tau',\eta>0$ depending only on $\tau$ such that
$(r_j (1-t_j^2)^{1/2})_{1\le j\le 3}$
is strictly admissible for all $\bt=(t_1,t_2,t_3)$
satisfying $|t_i|\le\eta$ for each index $i$.
If $\delta$ is sufficiently small then for the vast majority
of all such $\bt$ that also satisfy $\sum_{j=1}^3 r_j t_j=0$,
the three slices $E_j^{(r_j t_j)}$ nearly coincide with homothetic ellipsoids
in the sense that $|E_j^{(r_j t_j)}\symdif r_j (1-t_j^2)^{1/2}\scripte(\bt)|= 0_\delta(1)$.
Moreover, $|\scripte(\bt)|\equiv \omega_{d-1}$.
By fixing a typical value $\bar t_1$ of $t_1$ and letting $t_2,t_3$ vary we
conclude that the ellipsoids $\scripte(\bar t_1,t_2,t_3)$
nearly coincide for nearly all $(t_2,t_3)$ satisfying $r_1\bar t_1+r_2 t_2+r_3 t_3=0$
and $|t_i|\le\eta$.
By interchanging the roles of the indices we conclude via transitivity that
$\scripte(\bt)$ nearly coincides with $\scripte(\bt')$
for the vast majority of all ordered pairs ($\bt,\bt'$) satisfying $|t_j|,|t'_j|<\eta$,
$\sum_j r_j t_j=0$, and $\sum_j r_j t'_j=0$.
By fixing a typical value of $\bt'$ we reach the desired conclusion that 
the ellipsoids $\scripte(\bt)$ may all be taken to coincide with a single ellipsoid $\scripte$ 
--- but still under the restriction that $|t_j|\le\eta$ for $j\in\{1,2,3\}$.

The same reasoning as in the proof of Proposition~\ref{prop:lesssymmetrizedcompactness}
based on the approximate functional equation \eqref{eq:displacementscompatible}
proves that the vectors $v_j(t_j)$ take the form \[v_j(t_j)=  t_ju_j+w_j+o_\delta(1)\]
for all $t_j\in[-\eta,\eta]$ outside a set of Lebesgue measure $o_\delta(1)$,
where $u_j,w_j\in\reals^{d-1}$ and $\sum_{j=1}^3 r_j w_j= \sum_{j=1}^3 u_j=0$.
Therefore by transforming $\reals^d$ (separately for each index $j$) 
by an affine automorphism $(x';\,t)\mapsto (x'-tu_j-w_j;\,t)$
we may reduce to the case in which $v_j(t_j)\equiv 0$ for $|t_j|\le\eta$.

Thus far we have established two useful conclusions.

\begin{lemma}\label{lemma:gettingthere} For each $d\ge 2$ and $\tau>0$ 
there exists $\eta>0$ with the following property.
Let $\EE$ satisfy the hypotheses of Proposition~\ref{prop:finalmainstep} with $\max_j|E_j|=1$.
Then there exists an ordered triple $(T_1,T_2,T_3)\in\baff(d)$
of measure-preserving transformations of $\reals^d$ such that 
$(T_j(E_j))_{1\le j\le 3}$ continues to satisfy the hypotheses of
Proposition~\ref{prop:finalmainstep} and 
\begin{equation} |B(0,\eta)\setminus T_j(E_j)| = o_\delta(1)  \end{equation}
for each $j\in\{1,2,3\}$.  \end{lemma}

\begin{lemma}\label{lemma:gettingthere2}
Let $d\ge 2$ and $\tau>0$. 
Let $\EE$ satisfy the hypotheses of Proposition~\ref{prop:finalmainstep}.
For each $j\in\{1,2,3\}$ there exists a partition
$E_j=G_j\cup B_j$ of $E_j$
such that $|B_j|\le o_\delta(1)$ and for each $x'\in\reals^{d-1}$, 
either $|G_j^{x'}|=0$ or there exists an interval $J_{x'}\subset\reals$ such that 
\begin{equation}
\big|G_j^{x'}\symdif J_{x'}\big|\le o_\delta(1)|G_j^{x'}|.
\end{equation}
\end{lemma}

In this statement, the intervals $J_{x'}$ are permitted to depend on the indices $j$.
Lemma~\ref{lemma:gettingthere2} follows from \eqref{eq:nearlyloid} by
interchanging the roles of the first and the $d$--th coordinates,
and invoking Fubini's theorem.  \qed

Now let $\EE$ satisfy the normalization $\max_j|E_j|=1$ 
and the conclusion of Lemma~\ref{lemma:gettingthere},
in addition to the hypotheses of Proposition~\ref{prop:finalmainstep}.
Assume without loss of generality that $\max_{j\in\{1,2,3\}}|E_j|=1$.
The $\tau$--admissibility hypothesis guarantees that $|E_k|\ge\tau$ for all $k\in\{1,2,3\}$.

Denote elements of $\reals^d$ by $x=(x_1,x_2,\dots,x_d)$.
Consider any index $i\in\{1,2,\dots,d\}$,
and apply the first part of the above analysis to $\EE$,
with the roles of the $i$--th and the $d$--th coordinates interchanged.
Conclude --- without making any supplementary changes of variables ---
that for $j\in\{1,2,\dots,d\}$ there exist $s>0$, $c_{j}\in\reals$,
and $\alpha$,
each of which potentially depends also on the index $i$, such that for each 
$t\in[-s+c_{j},s+c_{j}]$,
\begin{equation} \label{eq:nearend;slicemeasures}
\big|\, |\{x\in E_j: x_i = t+c_j\}| - \alpha (s^2-t^2)^{(d-1)/2} \,\big| =o_\delta(1)
\end{equation}
except for a set of parameters $t\in[-s+c_j,s+c_j]$ having one-dimensional
Lebesgue measure $o_\delta(1)s$.
Moreover, 
\begin{equation}\label{eq:nearend;outside}
|\{x\in E_j: x_i\notin[-s+c_i,s+c_i]\}| = o_\delta(1)s,\end{equation}
and $\omega_d s\alpha^{d-1}=|E_j|$. 
The quantities $s,\alpha$ are independent of $j$, while $\sum_{j=1}^3 c_j=0$.

Since no changes of variables have been made, $E_j$ contains a $d$--dimensional
ball $B(0,\eta)$, whose radius $\eta>0$ depends only on the dimension $d$.
This together with \eqref{eq:nearend;slicemeasures} 
implies a lower bound $\alpha\gtrsim \eta$, 
while \eqref{eq:nearend;outside} forces a lower bound $s\gtrsim \eta$.
The relation $\omega_d s\alpha^{d-1}=|E_j|$ then forces upper bounds on both $s,\alpha$.
Likewise, \eqref{eq:nearend;outside} implies an upper bound on $|c_j|$.
We have proved:

\begin{lemma}\label{lemma:cubed}
For each $d\ge 2$ and $\tau>0$ 
there exists $\rho>0$ with the following property.
Let $\EE$ satisfy the hypotheses of Proposition~\ref{prop:finalmainstep}
with $\max_j|E_j|=1$,
and satisfy the conclusion of Lemma~\ref{lemma:gettingthere}.
If $\delta$ is sufficiently small then 
for each $j\in\{1,2,3\}$,
\begin{equation} |E_j\setminus B(0,\rho)| = o_\delta(1).\end{equation}
\end{lemma}

\begin{corollary} \label{cor:precompact}
Let $d\ge 2$ and $\tau>0$.
Let $(\EE_n)_{n\in\naturals} 
= (E_{n,1},E_{n,2},E_{n,3})_{n\in\naturals}$ be a sequence of $\tau$--admissible 
ordered triples of Lebesgue measurable subsets of $\reals^d$
satisfying $\max_{1\le j\le 3}|E_{n,j}|=1$. Suppose that
\begin{equation} \scriptt({\mathbf E_n}) \ge (1-\delta_n)\scriptt({\mathbf E_n^\Star})
\ \text{where $\lim_{n\to\infty}\delta_n=0$.}
\end{equation}
Then there exists a sequence $\mbf{T_n}=(T_{n,j})_{1\le j\le 3}$ of elements of $\baff(d)$
such that for each index $j\in\{1,2,3\}$, $T_{n,j}$ is measure-preserving and
the sequence of indicator functions $(\one_{T_{n,j}(E_{n,j})})_{n\in\naturals}$
is precompact in $L^1(\reals^d)$.  \end{corollary}

\begin{proof}
Choose $\mbf{T_n}$ as in Lemma~\ref{lemma:gettingthere}. 
By an argument given for this same purpose in \cite{christbmtwo},
Rellich's lemma, Lemma~\ref{lemma:gettingthere2}, Lemma~\ref{lemma:cubed},
and simple Fourier transform upper bounds all together establish precompactness
of the resulting normalized sequences of indicator functions.
\end{proof}

If $\mbf{E_n}$ satisfies the hypotheses of Corollary~\ref{cor:precompact} and 
$\one_{E_{n,j}}\to f_j$ in $L^1$ norm for each $j\in\{1,2,3\}$,
then $f_j=\one_{E_j}$ for Lebesgue measurable sets satisfying
$\tau\le|E_j|\le 1$, $\mbf{E}=(E_j)_{1\le j\le 3}$ is $\tau$--admissible,
and $\scriptt(\mbf{E}) = \scriptt(\mbf{E}^\Star)$ by continuity of $\scriptt$. 
By Burchard's theorem \cite{burchard}, $\EE$ is an ordered triple of homothetic ellipsoids. 
The $L^1$ convergence means precisely that $|E_{n,j}\symdif E_j|\to 0$. 

Corollary~\ref{cor:precompact} is an equivalent restatement of 
Proposition~\ref{prop:finalmainstep}. If  
the Proposition were false then there would exist a sequence
$(\EE_n)_{n\in\naturals}$ satisfying the hypotheses of Corollary~\ref{cor:precompact},
and $\eps>0$ independent of $n$,
such that for all $n$ and all ordered triples $(\scripte_1,\scripte_2,\scripte_3)$
of homothetic ellipsoids,
$\max_{1\le j\le 3} |E_j\symdif \scripte_j|\ge\eps$. 
This is a contradiction.

 \end{document}